\documentclass{article}

\usepackage{amsmath,amsfonts,amscd,amssymb,amsthm}
\usepackage{soul}
\usepackage{anyfontsize}
\usepackage{subcaption}
\numberwithin{equation}{section}
\usepackage{xfrac}
\usepackage{color}
\usepackage[pdfpagelabels=true,unicode=true]{hyperref}

\newtheorem{theorem}{Theorem}[section]
\newtheorem{lemma}[theorem]{Lemma}

\theoremstyle{remark}
\newtheorem{remark}[theorem]{Remark}

\newcommand{\mc}[1]{{\mathcal #1}}

\newcommand{\bb}[1]{{\mathbb #1}}

\usepackage[pdfpagelabels=true,unicode=true]{hyperref}
\hypersetup
{
	pdfauthor={Leandro Chiarini, Alexandre Stauffer},
	pdftitle={Absorbing-state phase transition and activated random walks with unbounded capacities},
	pdfkeywords={Activated Random Walks, Random Environments, absorbing-state phase transition},
	colorlinks=true,
	linkcolor=blue,
	citecolor=blue,
	filecolor=blue,
	urlcolor=blue
}
\definecolor{Red}{cmyk}{0,1,1,0}

\title{Absorbing-state phase transition and activated random walks with unbounded capacities}

\author{Leandro Chiarini\footnote{l.chiarinimedeiros@uu.nl; Utrecht University,Department of Mathematics, Budapestlaan 6, 3584 CD Utrecht, The Netherlands.}
  \and Alexandre Stauffer\footnote{astauffer@mat.uniroma3.it; Universit\`a Roma Tre, Dip.\ di Matematica e Fisica, Largo S.\ Murialdo 1, 00146, Rome, Italy; University of Bath, Dept.\ of Mathematical Sciences, BA2 7AY Bath, UK, supported by EPSRC Fellowship EP/N004566/1.}}

\begin{document}

\maketitle
\begin{abstract}
In this article, we study the existence of an absorbing-state phase transition of an Abelian process that generalises the Activated Random Walk (ARW).
Given a vertex transitive $G=(V,E)$, we associate to each site $x \in V$ a capacity  $w_x \ge 0$, which describes how many inactive particles $x$ can hold, where $\{w_x\}_{x \in V}$  is a collection of i.i.d random variables.
 When $G$ is an amenable graph, we prove that if $\bb E[w_x]<\infty$,  the model goes through an absorbing-state phase transition  and if $\bb E[w_x]=\infty$, the model fixates for all $\lambda>0$.
Moreover, in the former case, we provide bounds for the critical density that match the ones available in the classical Activated Random Walk.
\end{abstract}

\section{Introduction}
\label{sec:intro}

In this article, we study a generalised version of the Activated Random Walks (ARW) model, which was introduced in \cite{rolla2012absorbing}.
 The original ARW can be described as it follows: we sample a collection of randomly placed particles in a graph with density $\mu$.
Each of these particles evolves independently according to a continuous time random walk of rate $1$.
Moreover, each particle has an independent Poisson clock of rate $\lambda$  so that, when such a clock rings, the particle becomes inactive.
 If inactive and active particles share the same site, the inactive particle instantly become active.

The ARW has been extensively studied in a range of contexts, in particular in establishing the so-called \textit{absorbing-state phase transition} \cite{basu2019activated,klivans2018mathematics, stauffer2018critical,taggi2016absorbing, taggi2019active}.
Moreover, it falls in a wider class of particle systems called \textit{Abelian Networks}.
For such class of systems, the order of the rings of underlying Poisson clocks used to define the dynamics is to some extent irrelevant.
General properties of such class can be found in the series of papers \cite{bond2016abelian1, bond2016abelian2,bond2016abelian3}.
Besides the ARW, famous examples of Abelian Networks include the stochastic sandpile, bootstrap percolation, chip-firing dynamics,  oil-water dynamics,  and rotor-router models.

In the present article, we will study the following model.
Given an infinite and vertex transitive graph $G=(V,E)$, we introduce a random environment $w = \{w_x\}_{x \in V}$ given by a collection of random variables sampled independently according to a common measure $\nu$ supported in $\bb N:=\{0,1,\dots\}$.
Then, we sample independently an initial  configuration of active particles $\{\eta^a_0(x)\}_{x \in V}$ according to a Poisson measure of parameter $\mu$.
Again, each of these particles evolves independently according to a continuous time random walk of rate $1$.
Likewise, each particle has an independent Poisson clock of rate $\lambda$  so that, when such a clock rings, the particle becomes inactive.
However, in this model, whenever active and inactive particles share the same site $x$ all of them immediately become inactive if the number of particles at $x$ is smaller or equal to the site capacity $w_x$; otherwise, all of them immediately become active.
Notice that the classical ARW corresponds to the case $w_x = 1$ for all $x\in V$.
 Notice also that this model allows the case $w_x=0$, therefore there is no stochastic dominance between the system discussed here and the original ARW.

The reason to introduce such a model comes from one of the main characteristics of Abelian models, the aforementioned absorbing-state phase transition: a transition between a regime of \textit{fixation} and a regime of \textit{activity}.
We say that a model fixates if the configuration restricted to any finite  subgraph eventually becomes constant in time; otherwise, we say it stays active.
Certain models display a phase transition in the initial density $\mu$: for sufficiently small density $\mu$ the system fixates, but remains active for $\mu$ sufficiently large.
Both the ARW and the stochastic sandpile display phase transition in $\mu$; however, the oil-water dynamics does not.

An important property that distinguishes the ARW and the stochastic sandpile from the oil-water dynamics is that the former two models can only hold a (deterministic) maximum amount of particles at each site before becoming unstable.
Here, we will informally refer to this property as \textit{finite capacity}.
On the other hand, the oil-water dynamics can hold arbitrarily many particles at a stable vertex, hence, we say that it has \textit{infinite capacity}.
It is natural to wonder how connected the existence of absorbing-state phase transition in an Abelian model is to the finiteness of its capacity.

The activated random walks in random environments (ARWRE) studied in the present article interpolates between those two scenarios.
As each site $x$ is able to hold $w_x$ of inactive particles before becoming unstable, the model still has finite capacity in some sense.
However, the spatial averages of $w_x$ may converge or diverge depending on the existence of the first moment of $w_x$.
Therefore two natural questions arise, \textit{is it possible for a abelian process to have a absorbing-state phase transition with infinite capacity? Does finite capacity imply absorbing-state phase transition?} 

In the following, we show that the only important parameter for the existence of an absorbing-state phase transition in the ARWRE is the finiteness of the expected value of $w_x$.
 It is not clear whether the exact law of $w_x$ is relevant to determine the exact value of the critical density.

Moreover, it is interesting to point out that we do not take into consideration any geometrical information about the environment.
For instance, we still have no phase transition if $w$ is heavy-tailed but the site percolation given by ${\bf 1}_{[w_x=0]}$ is supercritical.

To state our main result, consider ${\bb P}^\nu_{\mu,\lambda}$ to
be the law of the whole process. 

\begin{theorem} \label{thm:main}
    Given an infinite, vertex transitive, and amenable graph $G=(V,E)$, 
    it follows that
    \begin{itemize}
      \item For all $\mu < \frac{\lambda}{1+\lambda}\bb E[w_o]$, we have
    ${\bb P}^\nu_{\mu,\lambda}\left( \text{The process } \eta_\cdot \text{ fixates} \right)=1$.
      \item For  $\mu > {\bb E} [w_o]$, we have
    ${\bb P}^\nu_{\mu,\lambda}\left(\text{The process } \eta_\cdot \text{ fixates} \right)=0$.
    \end{itemize}
    
    In particular, if $\bb E[w_0]=\infty$, the system fixates for all $\mu<\infty$
    and $\lambda>0$ with probability $1$. 
\end{theorem}

The upper bound is a very direct application of the mass-transport
techniques used in \cite{gurel2010fixation}. The lower bound comes 
from an adaptation of the concept of \textit{stabilisation via weak stabilisation} 
developed in \cite{stauffer2018critical}.

\section{Definitions and classical results}
\label{sec:defs}

Let $G=(V,E)$ be a vertex transitive and amenable graph with $d$ the common degree of its vertices.
For convenience, we will choose a vertex $o$ to call its origin.
We also will denote by $B(x,r)$ the ball of center in $x$ and radius $r$ in the graph distance.
We also write $x \sim y$ to denote that $x,y$ are neighbours in $G$.
Consider the state space
\[
\Omega
:=
\Big[
  \eta \in \bb N^V \times \bb N^V: \{x : \eta^a(x) \ge 1\}\cap \{x: \eta^i(x)\ge 1\}= \emptyset
\Big],
\]
where for any $\eta \in \Omega$, we denote
\[
  \eta=\left\{\left(\eta^{a}(x),\eta^{i}(x)\right) \right\}_{x \in V}.
\]
Here $\eta^a(x)$ denotes the number of active particles on the site $x $ and $\eta^i(x)$ is the number of inactive particles on the site $x$.

We consider the random environment.
That is, we sample an i.i.d collection $w=\{w_x\}_{x \in V} \in \bb N^V$ which represents the capacity of each site.
That is, how many, inactive particles can the site $x$ hold.
We will denote the law of the random variables $w_x$ by $\nu$.

Consider the collection of probability measures $({\bb P}_\mu)_{\mu >0}$ on $\Omega$ according to which the random variables $\eta^i(x)=0$ a.s for any $\mu>0$ and $\eta^a(x)$ are i.i.d Poisson random variables of parameter $\mu$.
This measure will be used to sample the initial condition of the system.

We say that a configuration is \textit{unstable} at the site $x$ if $\eta^a(x)>0$.
If the configuration $\eta$ is unstable at the site $x$, we define the \textit{sleep operator} $\tau^{x,s}$ which orders that all particles at the site $x$ sleep (or to turn inactive) at once, as long as $x$ has enough capacity to hold that many inactive particles.
Otherwise, all the particles on $x$ remain active.
More precisely, we define  
\[
  (\tau^{x,s} \eta)(z)
  =
  \begin{cases}
    (\eta^a(x),0) & \text{ if } z=x,\eta^a_x > w_x, \\
    (0,\eta^a(x)) & \text{ if } z=x,\eta^a(x) \le w_x, \\
    (\eta^a(z),\eta^i(z)) & \text{ if } z\neq x.
  \end{cases}
\]

If the configuration $\eta$ is unstable at the site $x$, we also define the \textit{jump operators} $\tau^{x,y}$ for $y \sim x$, which sends a particle from $x$  to $y$.
If $y$ already contained inactive particles, the new particle also becomes inactive.
At this point, if the number of inactive particles at $y$ surpasses the value of $w_y$ all of its particles will immediately turn active.
More precisely, we define $(\tau^{x,y} \eta)(z)$ to be 
\[
  (\tau^{x,y}\eta)(z):=
  \begin{cases}
    (\eta^a(x)-1,0) & \text{ if } z=x \\
    (\eta^a( y )+1,0) & \text{ if } z=y,\eta^i(y)=0, \\
    (0,\eta^i( y )+1) & \text{ if } z=y,\eta^i(y) \in \{1,\dots,w_{y}-1\}, \\
    (\eta^i( y )+1,0) & \text{ if } z=y,\eta^i(y)\ge w_{y}, \\
    (\eta^a(z),\eta^i(z)) & \text{ if } z\neq x,y.
  \end{cases}
\]

For $\lambda>0$, the evolution of the system is defined as the continuous time Markov chain given in the following.
A configuration $\eta$ transitions to the state $\tau^{x,y} \eta$ at rate $\eta^a(x)/d$ for each $y\sim x$, and transitions to the state $\tau^{x,s}\eta$ with rate $\lambda \eta^a(x)$.
 We will denote the law of the process $\eta_t$ by $\bb P^\nu_{\mu,\lambda}$. 

We will use the so-called \textit{Diaconis-Fulton representation}.
To construct it, we start with a fixed set of instructions 
$\tau \in \mc T$, where 
\[
  \mc T:=
  \left\{\tau^{x}_j \in \{\tau^{x,s},\tau^{x,y}, y \sim x\} ,x \in V, j\ge 1\right\}.
\]
In order to apply the instructions given by a stack $\tau \in \mc T$ at
a specific site $x \in V$ (which will also refer as \emph{toppling the 
site $x$}), we introduce a function which counts
how many times each site was toppled. With this purpose, let 
${\bf h}:=\{h_x\}_{x\in V}$ with $h_x \ge 0$.  
Then add ${\bf h}$ to the state of the system
and extend the action of $\tau^{x}_j$ to $(\eta,{\bf h})$ at the site $x$ as
\begin{equation*}\label{eq:def-top}
	\Phi_x\left(\eta, {\bf h}\right) 
	:=
	(\tau^{x}_{h_x+1}, {\bf h} + \delta_x),
\end{equation*}
where $\delta_{x,y}=1$ if $x=y$ and $0$ otherwise.
We will often simply write $\Phi_\alpha(\eta)$ to denote 
$\Phi_\alpha(\eta,0)$.

Notice that $\Phi_x$ is only well-defined if $\eta$ is unstable at $x$, in which case we refer to the map $\Phi_x$ being a \textit{legal toppling} for the configuration $(\eta, {\bf h})$.
 Moreover, for a sequence $\alpha=(x_1,\dots,x_l)$ we define a toppling sequence $\Phi_\alpha = \Phi_{x_l}\circ \dots \circ \Phi_{x_1}$ if for any $i \in \{1,\dots,l\}$, $\Phi_{x_i}$ is legal for the configuration $\Phi_{x_{i-1}}\circ \dots \circ \Phi_{x_1}(\eta,{\bf h})$, in which case we say that the sequence $\alpha$ is legal.
If either a toppling or a toppling sequence is not legal, we say it is \textit{illegal}.
We will also keep track of how many times a toppling sequence $\alpha=(x_1,\dots,x_l)$ used instructions of each stack, for this, we define
\[
  m_{\alpha}(x) := \sum_{j=1}^l {\bf 1}[x_j=x].
\]

Given a set $K \subset V$, we say that \textit{$\eta$ is stable in $K$}  if $\eta$ has no unstable sites in $K$.
We then extend this notion to toppling sequences, we say that a sequence \textit{$\alpha$ stabilises  $\eta$}  in $K$ if $\alpha$ is legal and $\Phi_\alpha \eta$ is stable in $K$.

For the sake of conciseness, we will only quickly state some standard properties of the model, that is, Abelian Property, monotonicity, and Zero-One law.
The proofs of the first two properties can be adapted from their counterparts in \cite{rolla2012absorbing} without significant changes.
The proof of the Zero-One Law only requires a small observation, which is given in Remark~\ref{rem:zero-one}.

Given a fixed stack of instructions $\tau$, and a finite set $K \subset V$, we can define a function $m_{K}=m^w_{K,\eta,\tau}$ to be equal to the number of times we use instructions at each site $x \in V$ until we stabilise $V$ conditioned on the environment $w$.
Lemma~\ref{prop:ab} guarantees that such function is well-defined.

\begin{lemma}[Abelian property]\label{prop:ab}
  Let $\alpha$ and $\beta$ be two legal sequences for $\eta$ such that $m_\alpha = m_\beta$,
  then we have that $\Phi_\alpha \eta \equiv \Phi_\beta\eta$. 
  Moreover, if $\alpha$ and $\beta$ are two legal toppling sequences 
  contained in $K \subset V$ and both stabilise $\eta$ in $K$, 
  then $m_\alpha= m_\beta$. 
\end{lemma}

\begin{remark} \label{rem:def-dynamics}
Our choice of dynamics might appear an unnatural a way of generalising the ARW dynamics.
One could feel tempted to changing the jumping operators so that whenever an active particle arrives at a site, it activates all the inactive particles therein regardless of how many particles where present there.
However, such dynamics do not satisfy the Abelian property.
See Figure~\ref{fig:counter-ex} below for an example.
\end{remark}

\begin{figure}[ht!]
    \centering
    \includegraphics[scale=0.5]{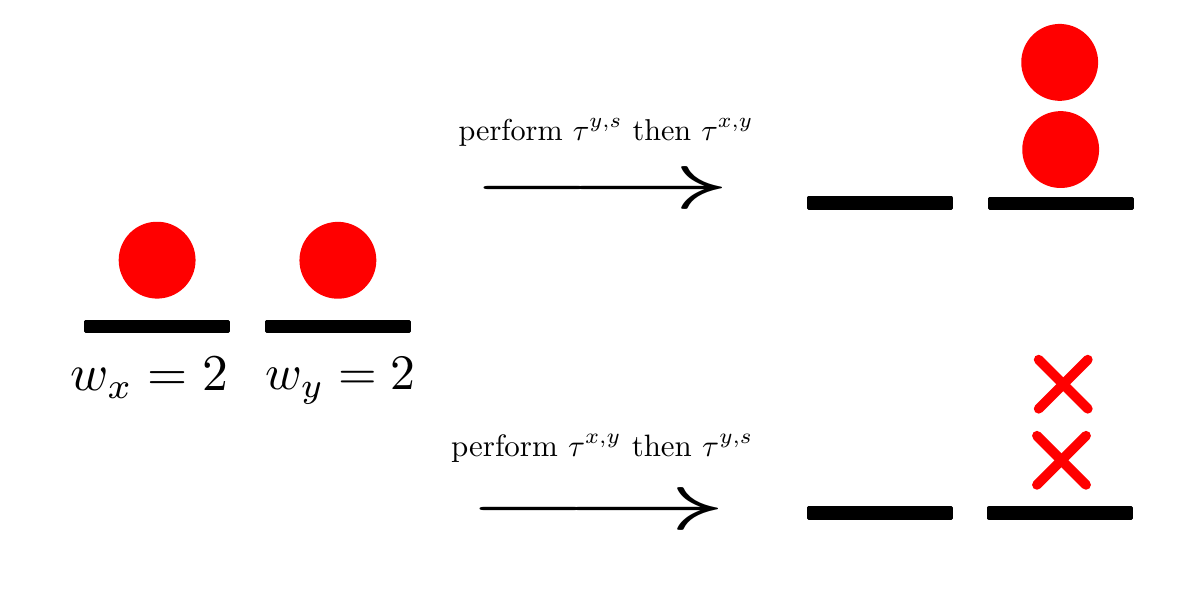}
    \caption{ An example of the non-Abealianess of the dynamics described in Remark~\ref{rem:def-dynamics}. The red balls
    represent active particles and the red crosses represent inactive particles.}
  \label{fig:counter-ex}
\end{figure}
Given two configurations 
$\eta,\tilde{\eta} \in \Omega$,
we write $\eta \le {\tilde{\eta}}$ to denote that for every
$x \in V$, we have 
\[
   \eta^a(x) \le {\tilde{\eta}}^a(x)
  \quad
  \text{ and }
  \quad
	\eta^a(x) + \eta^i(x)  
	\le
	\tilde{\eta}^a(x) + \tilde{\eta}^i(x).
\]
We also use $w \ge \tilde{w}$ to denote 
$w_x\ge \tilde{w}_x$ for every $x \in V$.
\begin{lemma}[Monotonicity] \label{prop:mono}
  Let $K_1 \subset K_2\subset V$ finite, $w \ge \tilde{w}$, and 
  $\eta$ and $\tilde{\eta}$ be two configurations such that $\eta \le \tilde{\eta}$,
  we have that 
  $
    m^w_{K_1,\eta,\tau}
    \le 
    m^{\tilde{w}}_{K_2,\tilde{\eta},\tau}.
    $
\end{lemma}
Lemma~\ref{prop:mono} implies 
that given any sequence $\{K_n\}_n$ that increases to 
$V$, the function defined pointwise by the limit 
$
  m^w_{\eta,\tau} 
:=
  \lim_{n \to \infty} m^w_{K_n,\eta,\tau}
  $
is also well-defined in $\bb N \cup \{\infty\}$.

The stack of instructions so far was taken to be arbitrary, 
but in order to stablish a relation to  our original process
$\eta_t$, we need to sample $\tau$ independently of $\eta_0$.
Let ${\bf P}_\lambda$ a probability measure on $\mc T$, such that
$\tau^{x,j}$ are i.i.d. For each $y \sim x$,  $\tau^{x}_j$ is equal to 
$\tau^{x,y}$ with probability $\frac{1}{d(1+\lambda)}$, and
equal to $\tau^{x,s}$ with probability $\frac{\lambda}{1+\lambda}$.
We will denote the joint law of $w,\eta,\tau$ as
${\bf P}^\nu_{\mu,\lambda}$. 

\begin{lemma}[Zero-One Law] \label{prop:01}
  For any $\lambda,\mu >0$ , we have that 
  \[
    {\bf P}^\nu_{\mu,\lambda}\left( m^w_{\eta,\tau}(0)< \infty\right)=
    {\bb P}^\nu_{\mu,\lambda}\left(\text{The process } \eta_\cdot \text{ fixates} \right)
    \in \{0,1\}.
  \]
\end{lemma}

We say that $(\mu,\lambda)$ is in activation phase if 
${\bb P}^\nu_{\mu,\lambda}\left(\text{The process } \eta_\cdot \text{ fixates} \right)=0$,
otherwise, we say that $(\mu,\lambda)$ is in the
fixation phase.

\begin{remark}\label{rem:zero-one}
In this remark we highlight on the small differences between 
the proof of Lemma~\ref{prop:01} in our context and the original 
ARW in \cite{rolla2012absorbing}. In the context of the original ARW,
the proof of the Zero-One Law is split in three parts
\begin{enumerate}
    \item Proving that ${\mc P}^\nu_\lambda\left( m_{\eta,\tau}(0)< \infty\right) \in \{0,1\}$.
    \item Providing a coupling between $ \lim_{t\to \infty} u_t$ and $m_{\eta,\tau}$,
    where $u_t(x)$ is the number of times that particles at the site
    $x$ either sleep or move before the time $t \in \bb R^+$, 
    according to the dynamics induced by ${\bb P}^\nu_{\mu,\lambda}$ 
    \item Showing that there are no blowups at finite time. 
\end{enumerate}

The proofs of steps $2$ and $3$ do not require additional ideas.
The proof of $1$ is also split into multiple parts: \emph{i)} if there exists a
vertex $x$, such that $m_{\eta,\tau}(x)=\infty$, then for any
other vertex $y$ we also have $m_{\eta,\tau}(y)=\infty$, \emph{ii)} use 
that this event is invariant according to the measure of 
${\bf P}^\nu_{\mu,\lambda}$ , and therefore must have
probability $0$ or $1$. 
In our case, even if 
$\bb E_\nu[w_0]= +\infty$,  for any realisation of $w$, any given finite path $\gamma$ that 
connects the sites $x$ and $y$ satisfies $\sum_{z \in \gamma} w_z <\infty$, hence we can still
use a Borel-Cantelli to prove that if $x$ topples infinitely many times, so
does $y$. After that,
we proceed as usual.
\end{remark}

We finish this section by briefly introducing the concept of 
weak stabilisation of a finite subset $K \subset V$ 
with respect to $x \in K$. This is a simple adaptation of 
Definitions 3.2 and 3.3 in \cite{stauffer2018critical}.
Notice that the equivalent notions of Abelian Property and Monotonicity 
also apply in this context whose proofs can also be adapted from 
\cite{stauffer2018critical} and will be omitted.

We say that a configuration $\eta$ is \emph{weakly} stable in a subset $K \subset V$ 
with respect to a vertex $x \in K$ if $\eta^a(x)  \leq w_x$ and 
$\eta^a(y) =0$ for all $y \in K \setminus\{x\}$ . 
For simplicity, we will say that $\eta$ is weakly stable for $(x, K)$.
Moreover, the 
\emph{weak stabilisation} (WS) of $(x, K)$ is a sequence of topplings 
of unstable sites of $K \setminus\{x\}$ and of topplings 
of $x$ whenever $x$ has at least $w_x+1$ active particles, 
until a weakly stable configuration for $(x, K)$ is obtained.
The order of the topplings of a weak stabilisation can be 
arbitrary. 
We will denote by $m_{x,K}(y)$ the number of times 
the site $y$ was toppled during the WS of $(x,K)$. 

\section{Proofs}

\subsection*{Activation for $\mu > \bb E[w_o]$ }

To show Activation, we can then use Theorem~1.1 of \cite{gurel2010fixation}, which studies a class of interacting particle systems for which both the site-wise construction and the particle-wise construction are well-defined.
In this context, the authors show that that site stabilisation implies particle stabilisation.
In simple words, this means that if $(\mu,\lambda)$ is in fixation phase, then each particle only moves for a finite amount of time.
Therefore, we can define the limiting position of a given particle in a measurable manner.
The proof that ARW (which generalises for the ARW with unbounded capacities) belongs to such class of interacting particle systems can be found in \cite[Theorem~6]{rolla2018nonfixation}.

Although the results in \cite{gurel2010fixation} is originally stated for interactions which are automorphism  invariant, one can easily adapt it to include our case in which the activated random walks depend on the environment $w$ which is not automorphism invariant itself, but whose law is.

Now, suppose that $(\mu,\lambda)$ with $\mu > \bb E_\mu[w_o]$
is in the fixation phase 
in sites, then it must also fixate in particles.
We can then define 
\[
  F(x,y):= \bb E_{\mu,\lambda}^\nu[X_{x,y}],
\]
where
\[
  X_{x,y}:=\text{number of particles that start in } x \text{ and fixate at }y.
\]
Notice that $F(x,y)$ is indeed invariant by automorphisms.
Moreover, we have that $ \sum_{y \in V} X_{x,y}= \eta^a_0(x)$ a.s in the event $\{\eta \text{ fixates}\}$ and that $\sum_{y \in V} X_{y,x} \le w_x$, as no more than $w_x$ particles can fixate at the site $x$.
Therefore, a simple application of the mass-transport principle implies that 
\[
  \mu =
  \bb E_{\mu,\lambda}^\nu[\eta_0(x)]
  = \sum_{y \in V} F(x,y)
  = \sum_{y \in V} F(y,x)
  \le 
  \bb E_{\mu,\lambda}^\nu[w_x]
  =
  \bb E_\nu[w_x],
\]
which gives the desired result.
\subsection*{Fixation for $\mu < \frac{\lambda}{1+\lambda}\bb E[w_o]$ }

In this section, we adapt the argument of  \cite{stauffer2018critical} using a particular strategy to show that if $(\mu,\lambda)$ is in the active phase, then $\mu \ge \frac{\lambda}{1+\lambda}\bb E[w_o]$.

For this proof, fix a monotone sequence of finite sets $\Lambda_n \subset V$, such that $\bigcup_{n \ge 1} \Lambda_n = V$ satisfying  $|\partial \Lambda_n|/|\Lambda_n|\to 0$ monotonically, where $\partial K$ denotes the external boundary of $K$.
Moreover, for $r \ge1$, let $\Lambda_{n,r} := \bigcup_{x \in \Lambda_n} B(x,r)$, notice that
\begin{equation}\label{eq:lambdarn}
  1 \le
  \frac{|\Lambda_{n,r}|}{|\Lambda_n|}
  \le
  1 + d^r \frac{| \partial \Lambda_n|}{ |\Lambda_n|},
\end{equation}
where $d$ is the degree of the vertices of $V$.

The objective is to prove that if we assume that $(\mu,\lambda)$ is in the activation phase, there exists a particular toppling strategy that allows us to obtain a lower bound on how many particles are inactive in the set $\Lambda_{n,r} $ after its stabilisation for $n$ sufficiently large.

This strategy is to stabilise  $\Lambda_n$ by performing a weak stabilisation of $(x,\Lambda_n)$ followed by a toppling of $x$, if such configuration is not stable in $\Lambda_n$, we repeat the previous two steps until we stabilise $\Lambda_n$.
This strategy for stabilisation will be referred to as \textit{stabilisation via weak stabilisation}.

In order to analyse such strategy, it is useful to understand how many 
particles are in the site $x$ after the weak stabilisation $(x,K)$. 
Therefore, we introduce the \emph{saturation event}.
For $K \subset V$ finite, $x \in K$. We define the event $S(x,K)$, 
\[
  S(x,K):= \left[\substack{\text{ After performing weak stabilisation }(x,K), 
  \\ \text{ we have exactly } w_x  \text{ active particles at }x}\right].
\]
We refer to the occurrence of the event $S(x,K)$ as \textit{the weak stabilisation $(x,K)$ saturates the site $x$}.

\begin{lemma}\label{lem:Sat-x}
  For $x \in K_1 \subset K_2$,  we get 
  \[
    S(x,K_1)
    \subset
    S(x,K_2).
  \]
  In addition, if $(\mu,\lambda)$ is in the activation phase, for each 
  $x \in V$, we have
  \[
    \lim_{r\to \infty}{\bf   P}^\nu_{\mu,\lambda}\Big(S(x,B(x,r))\Big)=1.
  \]
  Moreover, the convergence above is uniform in $x$. 
\end{lemma}
\begin{proof}
Let $\eta \in \Omega$, as the Abelian property still holds in the weak stabilisation setting, we can perform the WS in $(x,K_2)$ by first performing WS in $(x,K_1)$.
Denote by $\eta_1,\eta_2$, respectively, the configuration resulting from weakly stabilising $\eta$  in $(x,K_1)$ and $(x,K_2)$.
Note that, by the definition of weak stabilisation, if $S(x, K_1)$  holds, then $\eta^a_1(x) = w_x$.
This implies that  during the remaining topplings necessary to weakly stabilise $K_2$, site $x$  will not be toppled if it contains exactly $w_x$ active particles, implying that $\eta^a_2(x) \geq w_x$  and that $S(x, K_2)$  holds.

Now, we move to the second statement.
To simplify notation, we use that the event $S(x,B(x,r))$ is isomorphism invariant, so it is enough to examine the probability of $S(r):=S(o,B(o,r))$.

The first part of the proof implies that the limiting event $S(o,V):= \bigcup_{r \ge1}S(r)$ satisfies ${\bf P}^\nu_{\mu,\lambda} (S(o,V))= \lim_{r \to \infty}{\bf P}^\nu_{\mu,\lambda}(S(r))$.
Let 
\begin{equation}\label{eq:delta}
  \delta:={\bf P}^\nu_{\mu,\lambda}(S(o,V)^c),
\end{equation}
we will prove that $\delta=0$. 

Denote by $\tau^{*} \in \{\tau^{o,y}, y \sim o\}\cup \{\tau^{o,s}\}$ the first instruction in the stack of $o$ that was not used by the WS of $(o,B(o,r))$.
Due to  the monotonicity of $S(r)$, and the independence between $S(r)$ and $\tau^{*}$, we have that for all $r$ 
 \begin{equation}\label{eq:Bad-Event}
  \frac{\delta\lambda}{1+\lambda} 
  \le
  {\bf P}^\nu_{\mu,\lambda}\Big((S(r))^c \cap [\tau^{*}=\tau^{o,s}]\Big) 
  \le  {\bf P}^\nu_{\mu,\lambda}(A(r)),
 \end{equation}
where $A(r):=\left[m_{B(o,r)}(o) \le 1\right]$.
Indeed, if $S(r)$ does not occur, then $o$ does not accumulate more than $w_o$ particles before achieving the WS of $(o,B(o,r))$, so no toppling of $o$ is performed during this process.
At this point, all sites $B(o,r)\setminus \{o\}$ are stable.
If $\tau^*=\tau^{o,s}$, we have that the stabilisation via weak stabilisation only topples the site $o$ once.

 As the random variable $m_{B(o,r)}(o)$ is monotone in $r$, we have that 
 \[
  \lim_{r\to \infty} {\bf P}^\nu_{\mu,\lambda} (A(r))
  =
  {\bf P}^\nu_{\mu,\lambda} ([m_V(o)\le 1])=0,
 \]
where in the last equality, we used that $(\mu,\lambda)$ is in the activation phase.
However, using \eqref{eq:Bad-Event} and $\lambda >0$ , we get that $\delta=0$, which concludes the proof.
\end{proof}

The saturation event is useful to give a lower bound on the 
quantity of inactive particles that end in the site $x$. Indeed,
consider the quantity
\[
  Q(x,K)= Q(x,K,\mu,\lambda,\nu):=
  {\bf E}^{\nu}_{\mu,\lambda}\left[\eta^i_{\tau_{K}}(x) \right],
\]
where $\eta_{\tau_{K}}$ is the configuration obtained by stabilising 
$\eta$ in $K$.

We can then apply stabilisation via weak stabilisation to get 
the following lower bound.
\begin{lemma}\label{lem:NumberSleepPart}
  For any translation invariant, ergodic measure $\nu$ in $V^{\bb N}$ such that $\bb E_\nu[w_o]<\infty$,
  we have that 
  \begin{equation}\label{eq:NumberSleepPart}
    Q(x,K)
    \ge
   \frac{\lambda}{1+\lambda} {\bf E}^\nu_{\mu,\lambda}\Big[w_x\cdot {\bf 1}[S(x,K)]\Big].
  \end{equation}
\end{lemma}
The advantage of this lower bound is that, although $Q(x,K)$ is not monotone in $K$, the event $S(x,K)$ is.

\begin{proof}
We have that 
\[
Q(x,K)\ge
{\bf E}^{\nu}_{\mu,\lambda}\left[\eta^i_{\tau_{K}}(x) {\bf 1}[S(x,K)]\right].
\]
Notice that, in the event $S(x,K)$, after performing WS of $\eta$ in $(x,K)$, there are  $w_x$ particles at the site $x$.
Write $\tau^{x,*}$ to denote the first instruction at $x$ after performing WS in $(x,K)$.

If $S(x,K)$ occurs and $\tau^{x,*}=\tau^{x,s}$, the configuration $\eta$ will stabilise with exactly $w_x$ particles at the site $x$.
Again, as $\tau^{x,*}$ is independent of the state obtained by the WS of $\eta$ in $(x,K)$, we conclude the argument.
\end{proof}

Now, we assume that $(\mu,\lambda)$ with $\lambda>0$ is in activation phase.
\begin{equation}\label{eq:fix-r}
  {\bf E}^\nu_{\mu,\lambda}\Big[w_x\cdot {\bf 1}[S(x,B(x,r))]\Big]
  \ge
  \alpha \bb E [w_0].
\end{equation}

In the following, we write $\eta^a(K):= \sum_{x \in K} \eta^a(x)$ and $\eta^i(K):= \sum_{x \in K} \eta^i(x)$.
Notice that a.s, we have that $\eta^a_{0}(K)\ge \eta^i_{\tau_{K}}(K)$ for any finite set $K$ .
Therefore,
\begin{equation}\label{eq:fix-1}
  {\bf E}^\nu_{\mu,\lambda}\left[\eta^a_{0}\left(\Lambda_{n,r}\right)\right]
  \ge 
  {\bf E}^\nu_{\mu,\lambda}\left[\eta^i_{\tau_{\Lambda_{n,r}}}\left(\Lambda_{n,r}\right)\right].
\end{equation}
We will estimate both sides of \eqref{eq:fix-1}. 
We start with the right-hand side,
\begin{align*}
  {\bf E}^\nu_{\mu,\lambda}\left[\eta^i_{\tau_{\Lambda_{n,r}}}(\Lambda_{n,r})\right]
  \ge
  {\bf E}^\nu_{\mu,\lambda}\left[\eta^i_{\tau_{\Lambda_{n,r}}}\left(\Lambda_{n}\right)\right]
   &= 
   \sum_{x \in \Lambda_{n}}
  {\bf E}^\nu_{\mu,\lambda}\left[\eta^i_{\tau_{\Lambda_{n,r}}}(x)\right]
  \\ &= 
  \sum_{x \in \Lambda_{n}}
  Q\left(x,\Lambda_{n,r}\right).
\end{align*}
Using the isomorphism invariance, \eqref{eq:NumberSleepPart}, and \eqref{eq:fix-r} 
we get that 
\begin{align}\label{eq:fix-2}
  {\bf E}^\nu_{\mu,\lambda}\left[\eta^i_{\tau_{\Lambda_{n,r}}}(\Lambda_n)\right]
  &\ge
  \frac{\lambda}{1+\lambda}
  \sum_{x \in \Lambda_{n}}{\bf E}^{\nu}_{\mu,\lambda} \left[w_x {\bf 1}[S(x,\Lambda_{n,r})]\right]
  \\ \nonumber &\ge
  \frac{\lambda}{1+\lambda}
  \sum_{x \in \Lambda_{n}}  {\bf E}^{\nu}_{\mu,\lambda}\left[w_x {\bf 1}[S(x,B(x,r))]\right]
  \ge
  |\Lambda_{n}|\frac{\lambda}{1+\lambda} \alpha \bb E[w_0] 
\end{align}
Applying \eqref{eq:fix-2} in \eqref{eq:fix-1}, and using that the left-hand side
of \eqref{eq:fix-1} is equal to $\mu |\Lambda_{n,r}|$, we get
\[
  \mu \ge \frac{|\Lambda_{n}|}{|\Lambda_{n,r}|} \alpha \bb E[w_0].
\]
Finally, by taking $n \to \infty$ and using \eqref{eq:lambdarn}, we get
\[
  \mu \ge \alpha \frac{\lambda}{1+\lambda} \bb E[w_0].
\]
By taking $\alpha \rightarrow 1^{-}$ we conclude the proof.

\bibliographystyle{abbrv}
\bibliography{library.bib}

\end{document}